\documentclass[12pt]{amsart}
\usepackage{amsmath,latexsym,amscd,amsbsy,amssymb,amsfonts,amsthm,fleqn,leqno,
euscript, graphicx, texdraw, pb-diagram}
\usepackage[matrix,arrow,curve]{xy}
\numberwithin{equation}{section}

\newtheorem{thm}{Theorem}[section]
\newtheorem{pro}[thm]{Proposition}
\newtheorem{lem}[thm]{Lemma}
\newtheorem{cor}[thm]{Corollary}

\def\con{\subset}

\def\leukfrac#1/#2{\leavevmode
               \kern.1em
                \raise.9ex\hbox{\the\scriptfont0 ${}_#1$}
                \hskip -1pt\kern-.1em
                /\kern-.15em\lower.10ex\hbox{\the\scriptfont0 ${}_#2$}}

\theoremstyle{definition}

\newtheorem{remark}[thm]{Remark}
\newtheorem{que}[thm]{Question}
\theoremstyle{remark}

\makeatletter
\def\bd{\mathop{\operator@font bd}\nolimits}
\def\diam{\mathop{\operator@font diam}\nolimits}
\makeatother

\begin{document}


\title[On quasi $\kappa$-metrizable spaces]
{On quasi $\kappa$-metrizable spaces}

\author{Vesko  Valov}
\address{Department of Computer Science and Mathematics, Nipissing University,
100 College Drive, P.O. Box 5002, North Bay, ON, P1B 8L7, Canada}
\email{veskov@nipissingu.ca}
\thanks{Research supported in part by NSERC Grant 261914-13}




\begin{abstract}
The aim of this paper is to investigate the class of quasi $\kappa$-metrizable spaces. This class is invariant with respect to arbitrary products and contains Shchepin's \cite{sc1} $\kappa$-metrizable spaces as a proper subclass.
\end{abstract}

\maketitle

\markboth{}{Quasi $\kappa$-metrics}



\section{Introduction}
 Recall that a {\em $\kappa$-metric} \cite{sc1} on a space $X$  is a non-negative function $\rho(x,C)$ of two variables, a point $x\in X$ and a regularly closed set $C\subset X$, satisfying the following conditions:
\begin{itemize}
\item[K1)] $\rho(x,C)=0$ iff $x\in C$;
\item[K2)] If $C\subset C'$, then $\rho(x,C')\leq\rho(x,C)$ for every $x\in X$;
\item[K3)] $\rho(x,C)$ is continuous function of $x$ for every $C$;
\item[K4)] $\rho(x,\overline{\bigcup C_\alpha})=\inf_{\alpha}\rho(x,C_\alpha)$ for every increasing transfinite family $\{C_\alpha\}$ of regularly closed sets in $X$.
\end{itemize}
A $\kappa$-metric on $X$ is said to be {\em regular} if it satisfy also next condition
\begin{itemize}
\item[K5)] $\rho(x,C)\leq \rho(x,C')+\overline\rho(C',C)$ for any $x\in X$ and any two regularly closed sets $C,C'$ in $X$, where
$\overline\rho(C',C)=\sup\{\rho(y,C):y\in C'\}$.
\end{itemize}
We say that a function $\rho(x,C)$ is an {\em quasi $\kappa$-metric} (resp., a regular quasi $\kappa$-metric)  on $X$ if it satisfies the axioms $K2) - K4)$ (resp., $K2) - K5))$ and the following one:
\begin{itemize}
\item[K$1)^*$] For any $C$ there is a dense open subset $V$ of $X\setminus C$ such that $\rho(x,C)>0$ iff $x\in V$.
\end{itemize}
Obviously, we can assume that $\rho(x,C)\leq 1$ for all $x$ and $C$, in such a case we say that $\rho$ is a normed quasi $\kappa$-metric.

Quasi $\kappa$-metrizable spaces were introduced in \cite{v}. Our interest of this class was originated by Theorem 1.4 from \cite{v} stating that a compact space is quasi $\kappa$-metrizable if and only if it is skeletally generated.
Unfortunately, the presented there proof of the implication that any skeletally generated compactum is quasi $\kappa$-metrizable is not correct. Despite of this incorrectness, the class of
quasi $\kappa$-metrizable is very interesting. It is closed with respect to arbitrary products and contains as a proper subclass the $\kappa$-metrizable spaces. The aim of this paper is to investigate the class of quasi $\kappa$-metrizable spaces, and to provide a correct characterization of skeletally generated spaces.

The class of skeletally generated spaces was introduced in \cite{vv}. According to \cite[Theorem 1.1]{v}, a space is skeletally generated iff it is $I$-favorable in the sense of \cite{dkz}. Recall that a map $f:X\to Y$ is skeletal if $\mathrm{Int}\overline{f(U)}\neq\varnothing$) for every open $U\subset X$.
A space $X$ is skeletally generated \cite{vv} if there is an inverse system $\displaystyle S=\{X_\alpha, p^{\beta}_\alpha, A\}$ of
separable metric spaces $X_\alpha$ and skeletal surjective bounding maps $p^{\beta}_\alpha$
satisfying the following conditions:
$(1)$ the index set $A$ is $\sigma$-complete (every countable chain in $A$ has a supremum in $A$);
$(2)$ for every countable chain $\{\alpha_n\}_{n\geq 1}\subset A$ with $\beta=\sup\{\alpha_n\}_{n\geq 1}$ the space $X_\beta$ is a (dense)
subset of $\displaystyle\lim_\leftarrow\{X_{\alpha_n},p^{\alpha_{n+1}}_{\alpha_n}\}$;
$(3)$ $X$ is embedded in $\displaystyle\lim_\leftarrow
S$ and $p_\alpha(X)=X_\alpha$ for each $\alpha$, where $p_\alpha\colon\displaystyle\lim_\leftarrow S\to X_\alpha$ is the $\alpha$-th limit projection.
If in the above definition  all bounding maps $p^{\beta}_\alpha$ are open, we say that $X$ is openly generated.

All topological spaces are Tychonoff and the single-valued maps are continuous. 
The paper is organized as follows:
Section 2 contains the proof that any product of quasi $\kappa$-metrizable spaces is also quasi $\kappa$-metrizable. In Section 3 we provide some additional properties of quasi $\kappa$-metrizable spaces. For example, it is shown that this property is preserved by open and perfect surjections, and that the \v{C}ech-Stone compactification of any pseudocompact quasi $\kappa$-metrizable space is quasi $\kappa$-metrizable. The results from Section 3 imply that there exist quasi $\kappa$-metrizable spaces which are not $\kappa$-metrizable.  In Section 4 we introduce a similar wider class of spaces, the weakly $\kappa$-metrizable spaces, and proved that a compact space is skeletally generated if and only if it is weakly $\kappa$-metrizable. Hence,  every
skeletally generated space is also weakly  $\kappa$-metrizable. The converse implication is interesting for spaces with a countable cellularity only, but it is still unknown, see Questions 4.3 - 4.4.

\section{Products of quasi $\kappa$-metrizable  spaces}

Let $\mathcal{B}$ be a base for a space $X$ consisting of regularly open sets. A real-valued function $\xi:X\times\mathcal B\to [0,1]$
will be called a $\pi$-capacity if it satisfies the following conditions:
\begin{itemize}
\item[E1)] $\xi(x, U)=0$ for $x\not\in U$, and $0\leq\xi(x, U)\leq 1$ for $x\in U$.
\item[E2)] For any $U\in\mathcal B$ the set $\{x\in U:\xi(x, U)>0\}$ is dense in $U$.
\item[E3)] For any $U$ the function $\xi(x,U)$ is lower semi-continuous, i.e if $\xi(x_0,U)>a$ for some $x_0\in X$ and $a\in\mathbb R$, then there is a neighborhood $O_{x_0}$ with $\xi(x,U)>a$ for all $x\in O_{x_0}$.
\item[E4)] For any two mappings $U:T\to\mathcal B$ and $\lambda:T\to X$, where $T$ is a set with an ultrafilter $\mathcal F$, such that the limit
$\displaystyle\widetilde\lambda=\lim_{\mathcal F}\lambda(t)$ exists and $\displaystyle\lim_{\mathcal F}\xi(\lambda(t),U(t))> a > 0$, then there
exists $\widetilde U\in\mathcal B$ such that $\widetilde U\con\overline{\lim}_{\mathcal F}U(t)$ and $\xi(\widetilde\lambda,\widetilde U)>a$. Here,
$\overline{\lim}_{\mathcal F}U(t)=\bigcap_{F\in\mathcal F}\overline{\bigcup_{t\in F}U(t)}$.
\end{itemize}
A capacity is called regular if it satisfies also the following condition:
\begin{itemize}
\item[E5)] If $\xi(x,U)>a>0$, there exists $U_a\in\mathcal B$ such that $\xi(x,U_a)\geq a$  and $\xi(y,U)\geq\xi(x, U)-a$ for all $y\in U_a$.
\end{itemize}
Our definition of a $\pi$-capacity is almost the same as the Shchepin's definition \cite{sc2} of capacity, the only difference is that Shchepin requires
$\xi(x,U)>0$ for all $x\in U$.

\begin{lem}\label{lem1}
Suppose $\xi:X\times\mathcal B\to [0,1]$ is a (regular) $\pi$-capacity on $X$. Then the function $\rho_\xi(x,C)$,  $\rho_\xi(x,C)=0$ if $x\in C$ and $\rho_\xi(x,C)=\sup\{\xi(x,U):U\cap C=\varnothing\}$ otherwise, is a (regular) quasi $\kappa$-metric on $X$.
\end{lem}

\begin{proof}
Suppose $\xi$ is a $\pi$-capacity on $X$. Clearly, $\rho_\xi$ satisfies condition $K2)$. According to the proof of \cite[Proposition 6, chapter 3]{sc1} $\rho_\xi$ also satisfies conditions $K3) - K4)$. To check condition $K1)^*$, let $C$ be a proper regularly closed subset of $X$. Then there is a subfamily $\{U_\alpha:\alpha\in A\}$ of $\mathcal B$ covering $X\setminus C$. For every $\alpha$ the set $V_\alpha=\{x\in U_\alpha:\xi(x,U_\alpha)>0\}$ is dense in $U_\alpha$. So  $V=\bigcup_{\alpha\in A}V_\alpha$ is dense in $X\setminus C$ and $\rho_\xi(x,C)>0$ for all $x\in V$.

Let show that if $\xi$ is a regular $\pi$-capacity, then $\rho_\xi$ satisfies condition $K5)$. Suppose $C, C'$ are two regularly closed subsets of $X$ and $x\in X$. Obviously, $\rho_\xi(x,C)\leq\rho_\xi(x,C')$ implies $\rho_\xi(x,C)\leq\rho_\xi(x,C')+\overline{\rho_\xi}(C',C)$. So, let $\rho_\xi(x,C)>\rho_\xi(x,C')$,  and choose an integer $m$ such that $\rho_\xi(x,C)>\rho_\xi(x,C')+1/n$ for all $n\geq m$. Hence, there is $U\in\mathcal B$ such that $U\cap C=\varnothing$ and $\xi(x,U)>a_n=\rho_\xi(x,C')+1/n$. So, according to condition $E5)$, there is $U_{a_n}\in\mathcal B$ such that
$\xi(x,U_{a_n})\geq a_n$ and $\xi(x,U)\leq\xi(y,U)+a_n$ for all $y\in U_{a_n}$. Since $\xi(x,U_{a_n})\geq a_n$, $U_{a_n}\cap C'\neq\varnothing$ (otherwise we would have $\rho_\xi(x,C')\geq\rho_\xi(x,C')+1/n$). Hence, $\xi(x,U)\leq\xi(z,U)+a_n$ for every $z\in U_{a_n}\cap C'$, which yields
$\xi(x,U)\leq\overline{\rho_\xi}(C',C)+\rho_\xi(x,C')+1/n$ for all $n\geq m$ and  $U\in\mathcal B$ with $U\cap C=\varnothing$. Therefore,
$\rho_\xi(x,C)\leq\rho_\xi(x,C')+\overline{\rho_\xi}(C',C)$.
\end{proof}

\begin{lem}\label{lem2}
Let $\rho$ be a (regular) normed quasi $\kappa$-metric on $X$ and $\mathcal B$ be the family of all regularly open subsets of $X$. Then the formula $\xi(x,U)=\sup\{\rho(x,C):C\cup U=X\}$ defines a (regular) $\pi$-capacity on $X$.
\end{lem}

\begin{proof}
It is easy to show that $\xi$ satisfies conditions $E1)$ and $E3)$. Condition $E4)$ was established in \cite[Lemma 3]{sc2} in the case $\rho$ is a $\kappa$-metric, but the same proof works for quasi $\kappa$-metrics as well. Let show condition $E2)$. We fix $U\in\mathcal B$ and consider the family
$\mathcal B_U=\{G\in\mathcal B:G\varsubsetneq U\}$. For every $G\in\mathcal B_U$ the set $V_G=\{x\in G:\rho(x,C_G)>0\}$ is open and non-empty, where $C_G=X\setminus G$. Hence, $V=\bigcup_{G\in\mathcal B_U}V_G$ is dense in $U$. Moreover, $\xi(x,U)\geq\rho(x,C_G)$ for every $G\in\mathcal B_U$ and $x\in V_G$ because $U\cup C_G=X$. So, $\xi(x,U)>0$ for all $x\in V$.

It remains to show that $\xi$ is regular, i.e. it satisfies $E5)$, provided $\rho$ is regular. Let $U\in\mathcal B$ and $\xi(x,U)>a>0$ for some $x\in U$. 
Since $C_U=X\setminus U$ is the smallest regularly closed subset $C$ of $X$ with $U\cup C=X$, we have $\xi(z,U)=\rho(z,C_U)$ for all $z\in X$. 
The set $W=\{y\in X:\rho(y,C_U)>\rho(x,C_U)-a\}$ is open and non-empty because $\rho(\cdot,C)$ is continuous and $x\in W$. Let $U_a\in\mathcal B$ be the set $U_a=\rm{Int}\overline W$ and $C_{U_a}=X\setminus U_a$. Observe that $W\subset U$, which implies $U_a\subset U$ and $C_U\subset C_{U_a}$. 
Then, by $K5)$, $\rho(x,C_U)\leq\rho(x,C_{U_a})+\overline\rho(C_{U_a},C_U)$. Consequently,  $\rho(x,C_{U_a})\geq\rho(x,C_U)-\overline\rho(C_{U_a},C_U)$.
On the other hand, $\overline\rho(C_{U_a},C_U)=\sup_{y\in C_{U_a}}\rho(y,C_U)=\sup_{y\in C_{U_a}\setminus C_U}\rho(y,C_U)$ because $\rho(y,C_U)=0$ for all $y\in C_U$.  Observe also that $\rho(y,C_U)\leq\rho(x,C_U)-a$ for all $y\in C_{U_a}\setminus C_U$. So, $\overline\rho(C_{U_a},C_U)\leq\rho(x,C_U)-a$, and
thus $\xi(x,U_a)=\rho(x,C_{U_a})\geq a$. 
Finally, $\xi(y,U)=\rho(y,C_U)\geq\rho(x,C_U)-a=\xi(x,U)-a$ for all $y\in U_a$ because $U_a\con\overline W$. Hence, $\xi$ satisfies $E5)$.
\end{proof}

Let consider the following condition, where $\rho(x,C)$ is a non-negative function with $C$ being a regularly closed subset of $X$:
\begin{itemize}
\item[K$1)^{**}$] For any regularly closed $C\subsetneqq X$ there is $y\not\in C$ with $\rho(y,C)>0$  and $\rho(x,C)=0$ for all $x\in C$.
\end{itemize}
\begin{remark}
Observe that in the previous lemma we actually proved the following more general statement: Suppose $\rho$ satisfies conditions $K1)^{**}$ and $K2) - K4)$, and
$\rho(x,C)\leq 1$ for all $x\in X$ and all regularly closed sets $C\con X$. Then $\xi(x,U)=\sup\{\rho(x,C):C\cup U=X\}$ defines a  $\pi$-capacity on $X$. Moreover, $\xi$ is regular if $\rho$ satisfies also $K5)$.
\end{remark}

\begin{cor}\label{cor1}
Suppose there is a function $\rho$ on $X$ satisfying conditions $K1)^{**}$ and $K2) - K4)$. Then there is a quasi $\kappa$-metric $d$ on $X$. Moreover, $d$ is regular if $\rho$ satisfies also condition $K5)$.
\end{cor}

\begin{proof}
We can suppose that $\rho$ is normed. Then, by Lemma \ref{lem2}, there is a $\pi$-capacity $\xi$ on $X$. Finally, Lemma \ref{lem1}, implies the existence of a quasi $\kappa$-metric $d$ on $X$. Moreover, if $\rho$ satisfies condition $K5)$, then $\xi$ is regular, so is  $d$.
\end{proof}

\begin{thm}\label{thm1}
Any product of (regularly) quasi $\kappa$-metrizable spaces is (regularly) quasi $\kappa$-metrizable.
\end{thm}

\begin{proof}
Suppose $X=\prod_{\alpha\in A}X_\alpha$ and for every $\alpha$ there is a normed (regular) quasi $\kappa$-metric $\rho_\alpha$ on $X$. Following the proof of \cite[Theorem 2]{sc2}, for every $\alpha$ let $\mathcal T_\alpha$ be the family of all regularly open subsets of $X_\alpha$ and let $\mathcal B$ be the standard base for $X$ consisting of sets of the form $U=\bigcap_{i=1}^{n}\pi_{\alpha_i}^{-1}(U_i)$ with $U_i\in\mathcal T_{\alpha_i}$ and $U_i\neq X_{\alpha_i}$, where
$\pi_\alpha:X\to X_\alpha$ is the projection. Denote by $v(U)$ the collection $\{\alpha_1,...,\alpha_n\}$. According to Lemma \ref{lem2}, for every $\alpha$ there exists a (regular) $\pi$-capacity $\xi_\alpha$ on $X_\alpha$. Consider the function $\xi:X\times\mathcal B\to\mathbb R$ defined by
$\displaystyle\xi(x,U)=\inf_{\alpha\in v(U)}\xi_\alpha(\pi_\alpha(x),\pi_\alpha(U))/|v(U)|$. Obviously, condition $E1)$ is satisfied. Moreover, since for each $\alpha_i$ the set $W_i=\{z\in X_{\alpha_i}:\xi_i(z,U_i)>0\}$ is open and dense in $U_i$, the set $W=\bigcap_{i=1}^{n}\pi_{\alpha_i}^{-1}(W_i)$ is open and dense in $U$. 
So, $\xi$ satisfies condition $E2)$. Shchepin has shown that the function $\xi$ is a (regular) capacity provide each $\xi_\alpha$ is so, see the proof of \cite[Theorem 15]{sc1} and \cite[Theorem 2]{sc2}. The same arguments show that  $\xi$ also satisfies conditions $E3) - E4)$, and condition $E5)$ in case each $\xi_\alpha$ is regular. Therefore, $\xi$ is a (regular) $\pi$-capacity. Finally, by Lemma \ref{lem1}, there exists a (regular) quasi $\kappa$-metric on $X$.
\end{proof}


\section{Some more properties of quasi $\kappa$-metrizable spaces}

\begin{pro}
Let $X$ be a quasi $\kappa$-metrizable space and $Y\con X$. The $Y$ is also quasi $\kappa$-metrizable in each of the following cases: (i) $Y$ is dense in $X$; (ii) $Y$ is regularly closed in $X$; (iii) $Y$ is open in $X$.
\end{pro}
\begin{proof}
If $\rho$ is a quasi $\kappa$-metric on $X$ and $Y\con X$ is dense, the equality $d(y,\overline U^Y)=\rho(y,\overline U^X)$, where $U\con Y$ is open, defines
a quasi $\kappa$-metric on $Y$. The second case follows from the observation that every regularly closed subset of $Y$ is also regularly closed in $X$. The third case is a consequence of the first two because every open subset of $X$ is dense in its closure.
\end{proof}

Let consider the following condition.
\begin{itemize}
\item[K$4)^*$] $\rho(x,\overline{\bigcup C_n})=\inf_{n}\rho(x,C_n)$ for every increasing sequence $\{C_n\}$ of regularly closed sets in $X$.
\end{itemize}
\begin{lem}\label{lem3}
Suppose $X$ is a space admitting a non-negative function $\rho(x,C)$ satisfying conditions $K1)^{*}$, $K2)$, $K3)$ and $K4)^*$. Then $X$ is quasi $\kappa$-metrizable provided $X$ has countable cellularity. In particular, every compact space admitting such a function $\rho$ is quasi $\kappa$-metrizable.
\end{lem}
\begin{proof}
It suffices to show that $\rho$ satisfies condition $K4)$ in case $X$ is of countable cellularity, and this follows from \cite[Proposition 2.1]{ku1}. For reader's convenience we provide a proof. Let $\{C_\alpha\}_\alpha$ be an increasing transfinite family of regularly closed sets in $X$. Then $\overline{\bigcup_\alpha C_\alpha}=\overline{\bigcup_\alpha U_\alpha}$ and $\{U_\alpha\}_\alpha$ is also increasing, where $U_\alpha$ is the interior of $C_\alpha$.  Since $X$ has countable cellularity, there are countably many $\alpha_i$ such that
$\bigcup_{i\geq 1}U_{\alpha_i}$ is dense in $\bigcup_\alpha U_\alpha$. We can assume that the sequence $\{\alpha_i\}$ is increasing, so is the sequence $\{U_{\alpha_i}\}$. Because $\rho$ satisfies condition $K4)^*$, we have $\rho(x,\overline{\bigcup C_{\alpha_i}})=\inf_{i}\rho(x,C_{\alpha_i})$. This implies that $\rho(x,\overline{\bigcup C_\alpha})=\inf_{\alpha}\rho(x,C_\alpha)$. Indeed, since $\overline{\bigcup C_{\alpha_i}}=\overline{\bigcup C_{\alpha}}$,
$\inf_{\alpha}\rho(x,C_\alpha)<\inf_{i}\rho(x,C_{\alpha_i})$ for some $x\in X$ would implies the existence of $\alpha_0$ with $\rho(x,C_{\alpha_0})<\inf_{i}\rho(x,C_{\alpha_i})$ for all $i$. Because any two elements of the family $\{C_\alpha\}_\alpha$ are comparable with respect to inclusion, the last inequality means that $C_{\alpha_0}$ contains all $C_{\alpha_i}$. Hence, $C_{\alpha_0}=\overline{\bigcup_\alpha C_\alpha}$ and
$\rho(x,C_{\alpha_0})$ would be equal to $\inf_{i}\rho(x,C_{\alpha_i})$, a contradiction.

It was shown in \cite[Theorem 1.4]{v} that every compact space $X$ admitting a non-negative function $\rho(x,C)$ satisfying conditions $K1)^{*}$, $K2)$, $K3)$ and $K4)^*$ is skeletally generated, and hence $X$ has countable cellularity. Therefore, any such compactum is quasi $\kappa$-metrizable.
\end{proof}
It was shown by Chigogidze \cite{chi} that the \v{C}ech-Stone compactification of every pseudocompact $\kappa$-metrizable space is $\kappa$-metrizable.
We have a similar result for quasi $\kappa$-metrizable spaces.

\begin{thm}\label{thm2}
If $X$ is a pseudocompact (regularly) quasi $\kappa$-metrizable space, then $\beta X$ is (regularly) quasi $\kappa$-metrizable.
\end{thm}

\begin{proof}
Suppose $\rho(x,C)$ is a quasi $\kappa$-metric on $X$. We can assume that $\rho(x,\overline U^X)\leq 1$ for all $x\in X$ and all open $U\subset X$ ($\overline U^X$ denotes the closure of $U$ in $X$). For every open $W\subset\beta X$ consider the function $f_W$ on $X$ defined by $f_W(x)=\rho(x,\overline{W\cap X}^X)$. Let $\widetilde f_W:\beta X\to\mathbb R$ be the continuous extension of $f_W$, and define $d(y,\overline W)=\widetilde f_W(y)$, $y\in\beta X$. Obviously, $d(y,\overline W)=0$ if $y\in W\cap X$. Since $W\cap X$ is dense in $\overline W$, $d(y,\overline W)=0$ for all $y\in\overline W$. Moreover, if $\overline W\neq\beta X$, then $\overline{W\cap X}^X\neq X$. So, there is an open dense subset $V$ of $X\setminus\overline W$ with $\rho(x,\overline{W\cap X}^X)>0$ for all $x\in V$. Since $f_W$ is continuous, the set
$\widetilde V=\{y\in\beta X:f_W(y)>0\}$ is open in $\beta X$ and disjoint from $\overline W$. Finally, because $V\subset\widetilde V$ and $V$ is dense in
$X\setminus\overline W$, $\widetilde V$ is dense in $\beta X\setminus\overline W$.
 So, $d$ satisfies condition $K1)^{*}$. Conditions $K2)$ and $K3)$ also hold. Hence, by Lemma \ref{lem3}, it suffices to show that $d$ satisfies $K4)^*$. To this end, let $\{\overline W_n\}$ be an increasing sequence of regularly closed subsets of $\beta X$ and $W=\bigcup_{n\geq 1}W_n$. We have $d(y,\overline W)\leq\inf_nd(y,\overline W_n)$ for all $y\in\beta X$. Moreover, since $\rho$ satisfies $K4)$, $d(y,\overline W)=\inf_nd(y,\overline W_n)$ if $y\in X$. Suppose there is $y_0\in\beta X\setminus X$ with $d(y_0,\overline W)<\inf_nd(y_0,\overline W_n)$. Consequently, for every $n$ there exists a neighborhood $V_n$ of $y_0$ in $\beta X$ such that $\delta< d(y,\overline W_n)$ for all $y\in V_n$, where
$d(y_0,\overline W)<\delta<\inf_nd(y_0,\overline W_n)$. We also choose a neighborhood $V_0$ of $y_0$ with $d(y,\overline W)<\delta$ for all $y\in V_0$. This implies that $d(y,\overline W)<\delta\leq\inf_nd(y,\overline W_n)$ provided $y\in V=\bigcap_{n\geq 1}V_0\cap V_n$. But $V\cap X\neq\varnothing$ because $X$ is pseudocompact. Thus,
$d(y,\overline W)<\inf_nd(y,\overline W_n)$ for any $y\in V\cap X$, a contradiction.

It follows from the definition of $d$ that it satisfies condition $K5)$ provided $\rho$ is regular.
\end{proof}

\begin{cor}
Every pseudocompact quasi $\kappa$-metrizable space $X$ is skeletally generated.
\end{cor}

\begin{proof}
We already noted that every quasi $\kappa$-metrizable compactum is skeletally generated, see \cite{v}. So, by Theorem \ref{thm2}, $\beta X$ is skeletally generated. Finally, by \cite{dkz} and \cite{v}, every dense subset of a skeletally generated space is also skeletally generated.
\end{proof}

\begin{pro}\label{open image}
Suppose $f:X\to Y$ is a perfect open surjection and $X$ is (regularly) quasi $\kappa$-metrizable. Then $Y$ is also (regularly) quasi $\kappa$-metrizable.
\end{pro}

\begin{proof}
Let $\rho$ be a quasi $\kappa$-metric on $X$. Since $f$ is open, $f^{-1}(\overline U)=\overline{f^{-1}(U)}$ for any open $U\con Y$. So, $f^{-1}(\overline U)$ is regularly closed set in $X$ and we
define $$d(y,\overline U)=\sup\{\rho(x,f^{-1}(\overline U)):x\in f^{-1}(y)\}.$$
One can check that $d$ satisfies conditions $K2)$ and $K4)$, and condition $K5)$ in case $\rho$ is regular. Moreover, $\overline U\neq Y$ implies $f^{-1}(\overline U)\neq X$. So, there is a dense open subset $V\con X\setminus f^{-1}(\overline U)$ such that $\rho(x,f^{-1}(\overline U))>0$ iff $x\in V$. Then $f(V)$ is a dense and open subset of $Y\setminus\overline U$ such that $f^{-1}(y)\cap V\neq\varnothing$ for all $y\in f(V)$. Hence, $d(y,\overline U)>0$ if $y\in f(V)$. If $y\not\in f(V)$, then $f^{-1}(y)\cap V=\varnothing$. Thus, $d(y,\overline U)>0$ iff $y\in f(V)$. Finally, let check continuity of the functions $d(\cdot,\overline U)$. Suppose $d(y_0,\overline U)<\varepsilon$ for some $y_0$ and $U$. Then $\rho(x,f^{-1}(\overline U))<\varepsilon$ for all $x\in f^{-1}(y_0)$. Consequently, there is a neighborhood $W$ of $f^{-1}(y_0)$ with
$\rho(x,f^{-1}(\overline U))<\varepsilon$ for all $x\in W$. Since, $f$ is closed, $y_0$ has a neighborhood $G$ such that $f^{-1}(G)\subset W$. This implies that $d(y,\overline U)<\varepsilon$ for all $y\in G$. Now, let $d(y_0,\overline U)>\delta$ for some $\delta\in\mathbb R$. So, there exists $x_0\in f^{-1}(y_0)$ with $\rho(x_0,f^{-1}(\overline U))>\delta$. Choose a neighborhood $O$ of $x_0$ such that $\rho(x,f^{-1}(\overline U))>\delta$ for all $x\in O$. Then, $f(O)$ is a neighborhood of $y_0$ and $d(y,\overline U)>\delta$ for any $y\in f(O)$. Therefore, each $d(.,\overline U)$ is continuous.
\end{proof}

Recall that a surjective map $f:X\to Y$ is {\em irreducible} provided there is no a proper closed subset $F$ of $X$ with $f(F)=Y$.
\begin{pro}\label{irreducible}
Let $f:X\to Y$ be a perfect irreducible surjection, and $Y$ is (regularly) quasi $\kappa$-metrizable. Then $X$ is also (regularly) quasi $\kappa$-metrizable.
\end{pro}

\begin{proof}
Suppose $\rho$ is a quasi $\kappa$-metric on $Y$. For every regularly closed  $C\con X$ define $d(x,C)=\rho(f(x),f(C))$. This definition is correct because $f(C)$ is regularly closed in $Y$. Indeed, let $C=\overline U$ with $U$ open in $X$. Since $f$ perfect and irreducible, we have
$f(C)=\overline{U_\sharp}$, where $U_\sharp=\{y\in Y:f^{-1}(y)\con U\}=Y\setminus f(X\setminus U)$ is open in $Y$.
 It is easily seen that $d$ satisfies conditions $K2)$ and $K3)$. Condition $K4)$ follows from the equality
$f(\overline{\bigcup_\alpha C_\alpha})=\overline{\bigcup_\alpha f(C_\alpha)}$ for any family of regularly closed sets in $X$. To see that $d$ satisfies also condition $K1)^*$, we observe that for every regularly closed $C\subsetneqq X$ there is a dense open subset $V\con Y\setminus f(C)$ such that
$\rho(y,f(C))>0$ iff $y\in V$. Then $W=f^{-1}(V)$ is open in $X$ and disjoint from $C$. Moreover, $d(x,C)>0$ iff $x\in W$. It remains to show that $W$ is dense in $X\setminus C$. And that is really true because for every open $O\con X\setminus\ C$ the set $O_\sharp$ is a non-empty open subset of $Y\setminus f(C)$. So, $O_\sharp\cap V\neq\varnothing$, which implies $W\cap O\neq\varnothing$.

One can also see that $d$ is regular provided so is $\rho$.
\end{proof}
It is well known that for every space $X$ there is a unique extremally disconnected space $\widetilde X$ and a perfect irreducible map $f:\widetilde X\to X$.
The space $\widetilde X$ is said to be the {\em absolute of $X$}. A space $Y$ is called {\em co-absolute to $X$} if their absolutes are homeomorphic.
\begin{cor}
The absolute of any (regularly) quasi $\kappa$-metrizable space is (regularly) quasi $\kappa$-metrizable.
\end{cor}

\begin{remark}
The last corollary shows that the class of $\kappa$-metrizable spaces is a proper subclass of the quasi $\kappa$-metrizable spaces. Indeed, let $X$ be a $\kappa$-metrizable compact infinite space. Then its absolute $aX$ is  quasi $\kappa$-metrizable. On the other hand, $aX$ being extremally disconnected can not be $\kappa$-metrizable (otherwise, it should be discrete by \cite[Theorem 11]{sc1}).
\end{remark}

\begin{cor}
Every compact space co-absolute to a quasi $\kappa$-metrizable space is skeletally generated.
\end{cor}

\begin{proof}
Let $X$ and $Y$ be compact spaces having the same absolute $Z$. So, there are perfect irreducible surjections $g:Z\to Y$ and $f:Z\to X$. If $Y$ is quasi $\kappa$-metrizable, then so is $Z$, see Proposition \ref{irreducible}. Hence, $Z$ is skeletally generated, and by \cite[Lemma 1]{kp2}, $X$ is also skeletally generated.
\end{proof}
Recall that the hyperspace $\rm{exp}X$ consists of all compact non-empty subsets $F$ of $X$ such that the sets of the form
$$[U_1,..,U_k]=\{H\in\rm{exp}X: H\con\bigcup_{i=1}^{k} U_i{~}\mbox{and}{~}H\cap U_i\neq\varnothing{~}\mbox{for all}{~}i\}$$  form a base $\mathcal B_{\rm{exp}}$ for $\rm{exp}X$,
where each $U_i$ belongs to a base $\mathcal B$ for $X$, see \cite{k}.
\begin{pro}
If $X$ is (regularly) quasi $\kappa$-metrizable, so is $\rm{exp}X$.
\end{pro}
\begin{proof}
Let $\mathcal B$ be a base for $X$ and $\rho$ be a (regular) quasi $\kappa$-metric on $X$. Then $\rho$ generates a (regular) $\pi$-capacity $\xi_\rho:X\times\mathcal B\to\mathbb R$ on $X$.
Following the proof of \cite[Theorem 3]{sc2}, we define a function $\xi:\rm{exp}X\times\mathcal B_{\rm{exp}}\to\mathbb R$ by
$$\xi(F,[U_1,..,U_k])=\frac{1}{k}\min\{\inf_{x\in F}\max_i\xi_\rho(x,U_i),\min_{i}\sup_{x\in F}\xi_\rho(x,U_i)\}.$$
It was shown in \cite{sc2} that $\xi$ satisfies conditions $E1)$, $E3)$ and $E4)$, and that $\xi$ is regular provided $\xi_\rho$ is regular.
Let show that $\xi$ satisfies condition $E2)$. Since $\xi$ is satisfies $E3)$, it suffices to prove that for every $[U_1,..,U_k]$ there is a dense subset
$V_{\rm{exp}}\con [U_1,..,U_k]$ with $\xi(F,[U_1,..,U_k])>0$ for all $F\in V_{\rm{exp}}$. To this end, for each $i$ fix an open dense subset $V_i$ of $U_i$ such that $\xi_\rho(x,U_i)>0$ if $x\in V_i$. Let $V_{\rm{exp}}$ consists of all finite sets $F\con X$ such that $F\con\bigcup_{i=1}^nV_i$ and $F\cap V_i\neq\varnothing$ for all $i$. Then $V_{\rm{exp}}$ is dense in $[U_1,..,U_k]$ and $\xi(F,[U_1,..,U_k])>0$ for all $F\in V_{\rm{exp}}$. Hence, by Lemma \ref{lem1}, $\rm{exp}X$ is (regularly) quasi $\kappa$-metrizable.
\end{proof}

Shchepin \cite[Theorem 3a]{sc2} has shown that if $\rm{exp}X$ is $\kappa$-metrizable, then so is $X$. We don't know if a similar result is true for quasi $\kappa$-metrizable spaces. 

\section{Skeletally generated spaces}
In this section we provide a characterization of skeletally generated compact spaces in terms of functions similar to quasi $\kappa$-metrics.
We say that a non-negative function $d:X\times\mathcal C\to\mathbb R$ is a {\em weak  $\kappa$-metric}, where $\mathcal C$ is the family of all regularly closed subsets of $X$, if it satisfies conditions $K1)^*$, $K2) - K3)$ and the following one:
\begin{itemize}
\item[$K4)_{0}$] For every increasing transfinite family $\{C_\alpha\}_\alpha\con\mathcal C$ the function $f(x)=\inf_\alpha d(x,C_\alpha)$ is continuous.
\end{itemize}
\begin{thm}\label{skel}
A compact space is skeletally generated if and only if it is weakly  $\kappa$-metrizable.
\end{thm}

\begin{proof}
First, let show that every skeletally generated compactum $X$ is weakly  $\kappa$-metrizable. We embed $X$ as a subset of $\mathbb R^\tau$ for some cardinal $\tau$. Then, according to \cite[Theorem 1.1]{v}, there is a function $\rm e:\mathcal T_X\to\mathcal T_{\mathbb R^\tau}$ between the topologies of $X$ and
$\mathbb R^\tau$ such that: $(i)$ $\rm e(U)\cap\rm e(V)=\varnothing$ provided $U$ and $V$ are disjoint; $(ii)$ $\rm e(U)\cap X$ is dense in $U$. We define a new function $\mathrm e_1:\mathcal T_X\to\mathcal T_{\mathbb R^\tau}$,
$$\mathrm e_1(U)=\bigcup\{\mathrm e(V):V\in\mathcal T_X{~}\mbox{and}{~}\overline V\subset U\}.$$
Obviously $\mathrm e_1$ satisfies conditions $(i)$ and $(ii)$, and it is also monotone, i.e. $U\subset V$ implies
$\mathrm e_1(U)\subset\mathrm e_1(V)$. Moreover, for every increasing transfinite family $\gamma=\{U_\alpha\}$ of open sets in $X$ we have
$\mathrm e_1(\bigcup_\alpha U_\alpha)=\bigcup_\alpha\mathrm e_1(U_\alpha)$. Indeed, if $z\in\mathrm e_1(\bigcup_\alpha U_\alpha)$, then there is an open set
$V\in\mathcal T_X$ with $\overline V\subset\bigcup_\alpha U_\alpha$ and $z\in\mathrm e(V)$. Since $\overline V$ is compact and the family is increasing,
$\overline V$ is contained in some $U_{\alpha_0}$. Hence, $z\in\mathrm e(V)\subset\mathrm e_1(U_{\alpha_0})$. Consequently,
$\mathrm e_1(\bigcup_\alpha U_\alpha)\subset\bigcup_\alpha\mathrm e_1(U_\alpha)$. The other inclusion follows from monotonicity of $\mathrm e_1$.

Because $\mathbb R^\tau$ is $\kappa$-metrizable (see \cite{sc2}), there is a $\kappa$-metric $\rho$ on $\mathbb R^\tau$.
For every regularly closed $C\con X$ and $x\in X$ we can define the function $d(x,C)=\rho(x,\overline{\mathrm e_1(\rm{Int}C)})$, where $\overline{\mathrm e_1(U)}$ is the closure of $\mathrm e_1(U)$ in $\mathbb R^\tau$. It is easily seen that
$d$ satisfies conditions $K2) - K3)$. Let show that it also satisfies $K4)_0$ and $K1^*)$. Indeed, assume $\{C_\alpha\}$ is an increasing transfinite family of regularly closed sets in $X$. We put $U_\alpha=\mathrm{Int}C_\alpha$ for every $\alpha$ and $U=\bigcup_\alpha U_\alpha$.  Thus, $\mathrm e_1(U)=\bigcup_\alpha\mathrm e_1(U_\alpha)$. Since $\{\overline{\mathrm e_1(U_\alpha)}\}$ is an increasing transfinite family of regularly closed sets in $\mathbb R^\tau$, for every $x\in X$ we have
$$\rho(x,\overline{\bigcup_\alpha\mathrm e_1(U_\alpha)})=\inf_\alpha\rho(x,\overline{\mathrm e_1(U_\alpha)})=\inf_\alpha d(x,C_\alpha).$$
Hence, the function $f(x)=\inf_\alpha d(x,C_\alpha)$ is continuous on $X$ because so is $\rho(.,\overline{\bigcup_\alpha\mathrm e_1(U_\alpha)})$.
To show that $K1^*)$ also holds, observe that $d(x,C)=0$ if and only if
$x\in X\cap\overline{\mathrm e_1(\rm{Int}C)}$. 
Because $\mathrm e_1(\rm{Int}C)\cap X$ is dense in $C$, $C\con\overline{\mathrm e_1(\rm{Int}C)}$. Hence,
$V=X\setminus\overline{\mathrm e_1(\rm{Int}C)}$ is contained in $X\setminus C$ and $d(x,C)>0$ iff $x\in V$. To prove $V$ is dense in $X\setminus C$, let $x\in X\setminus C$ and $W_x\subset X\setminus C$ be an open neighborhood of $x$. Then $W\cap\rm{Int}C=\varnothing$, so $\mathrm e_1(W)\cap\mathrm e_1(\rm{Int}C)=\varnothing$. This yields
$\mathrm e_1(W)\cap X\subset V$. On the other hand, $\mathrm e_1(W)\cap X$ is a non-empty subset of $W$, hence $W\cap V\neq\varnothing$. Therefore, $d$ is a weak  $\kappa$-metric on $X$.

The other implication was actually established in the proof of Theorem 1.4] from \cite{v}, and we sketch the proof here. Suppose $d$ is a weak  $\kappa$-metric on $X$ and
embed $X$ in a Tychonoff cube $\mathbb I^A$ with uncountable $A$, where $\mathbb I=[0,1]$. For any countable set $B\con A$ let $\mathcal A_B$ be a countable base for $X_B=\pi_B(X)$ consisting of all open sets in $X_B$ of the form $X_B\cap\prod_{\alpha\in B} V_\alpha$, where each $V_\alpha$ is an open subinterval of $\mathbb I$ with rational end-points and $V_\alpha\neq\mathbb I$ for finitely many $\alpha$. Here $\pi_B:\mathbb I^A\to\mathbb I^B$ denotes the projection, and let $p_B=\pi_B|X$. For any open $U\subset X$ denote by $f_U$ the function $d(\cdot,\overline U)$. We also write $p_B\prec g$, where $g$ is a map defined on $X$, if there is a map $h:p_B(X)\to g(X)$ such that $g=h\circ p_B$. Since $X$ is compact this is equivalent to the following: if
$p_B(x_1)=p_B(x_2)$ for some $x_1,x_2\in X$, then $g(x_1)=g(x_2)$.  We say that a countable set $B\subset A$ is {\em $d$-admissible} if $p_B\prec f_{p_B^{-1}(V)}$ for every $V\in\mathcal A_B^{<\omega}$. Here $\mathcal A_B^{<\omega}$ is the family of all finite unions of elements from $\mathcal A_B$.
Denote by $\mathcal D$ the family of all $d$-admissible subsets of $A$. We are going to show that all maps $p_B:X\to X_B$, $B\in\mathcal D$, are skeletal and the inverse system $S=\{X_B: p^B_D:D\con B, D,B\in\mathcal D\}$ is $\sigma$-continuous. Since $\displaystyle X=\lim_\leftarrow S$, this would imply that $X$ is skeletally generated, see \cite{v} and \cite{vv}. 

Following the proof of \cite[Theorem 1.4]{v}, one can show that for any countable set $B\con A$ there is $D\in\mathcal D$ with $B\con D$, and the union of any increasing sequence of $d$-admissible sets is also $d$-admissible. So, we need to show only that $p_B:X\to X_B$ is a skeletal map for every $B\in\mathcal D$.
Suppose there is an open set $U\subset X$ such that the interior in $X_B$ of $\overline{p_B(U)}$ is empty. Then $W=X_B\setminus\overline{p_B(U)}$
is dense in $X_B$. Let $\{W_m\}_{m\geq 1}$ be a countable cover of $W$ with $W_m\in\mathcal A_B$ for all $m$. Since $\mathcal A_B^{<\omega}$ is finitely additive, we may assume that $W_m\subset W_{m+1}$, $m\geq 1$. Because $B$ is $d$-admissible, $p_B\prec f_{p_B^{-1}(W_m)}$ for all $m$. Hence, there are continuous functions
$h_m:X_B\to\mathbb R$ with $f_{p_B^{-1}(W_m)}=h_m\circ p_B$, $m\geq 1$. Recall that $f_{p_B^{-1}(W_m)}(x)=d(x,\overline{p_B^{-1}(W_m)})$ and
$p_B^{-1}(W)=\bigcup_{m\geq 1}p_B^{-1}(W_m)$. Therefore, $f_{p_B^{-1}(W)}(x)=d(x,\overline{p_B^{-1}(W)})\leq f(x)=\inf_m f_{p_B^{-1}(W_m)}(x)$ for all $x\in X$. Moreover, $f$ is continuous and
$f_{p_B^{-1}(W_{m+1})}(x)\leq f_{p_B^{-1}(W_m)}(x)$ because $W_m\subset W_{m+1}$. The last inequalities together with $p_B\prec f_{p_B^{-1}(W_m)}$
yields that $p_B\prec f$. So, there exists a continuous function $h$ on $X_B$ with $f(x)=h(p_B(x))$ for all $x\in X$. But $f(x)=0$ for all $x\in p_B^{-1}(W)$, so $f(\overline{p_B^{-1}(W)})=0$. This implies that $h(\overline W)=0$.
Since $p_B(\overline{p_B^{-1}(W)})=\overline W=X_B$, we have that $h$ is the constant function zero. Consequently, $f(x)=0$ for all $x\in X$. Finally,
the inequality $d(x,\overline{p_B^{-1}(W)})\leq f(x)$  yields that $d(x,\overline{p_B^{-1}(W)})=0$ for all $x\in X$.
On the other hand, $\overline{p_B^{-1}(W)}\cap U=\varnothing$. So, according to $K1^*)$, there is an  open subset $U'$ of $U$ with $d(x,\overline{p_B^{-1}(W)})>0$ for each $x\in U'$, a contradiction.
\end{proof}

Because any compactification of a skeletally generated space is skeletally generated (see \cite{v}) and the weakly  $\kappa$-metrizability is a  hereditary property with respect to dense subsets, we have the following

\begin{cor}
Every skeletally generated space is weakly  $\kappa$-metrizable.
\end{cor}

All results in Section 3, except Proposition 3.10, remain valid for weakly $\kappa$-metrizable spaces. Theorem \ref{skel} and a result of Kucharski-Plewik \cite[Theorem 6]{kp2} imply that Proposition 3.10 is also true for weakly $\kappa$-metrizable compacta. But the following questions are still open.
\begin{que}
Is any product of weakly $\kappa$-metrizable spaces weakly $\kappa$-metrizable?
\end{que}
If there exists a counter example to Question 4.3 which, in addition has a countable cellularity, then next question would have also a negative answer.
\begin{que}
Is any weakly $\kappa$-metrizable space with a countable cellularity  skeletally generated?
\end{que}

\smallskip
\textbf{Acknowledgments.} The author would like to express his gratitude to A. Kucharski for his careful reading and suggesting some improvements of the paper.


\end{document}